\documentclass[a4paper,12pt]{amsart}
\usepackage[utf8]{inputenc}
\usepackage[margin=30mm]{geometry}

\usepackage[backend=bibtex,style=alphabetic,doi=true,isbn=false,url=false,giveninits=true,maxbibnames=50]{biblatex}

\usepackage{amssymb,enumitem,graphicx,tikz,microtype,mathtools,mathdots}
\usetikzlibrary{matrix,backgrounds}
\usetikzlibrary{arrows.meta,positioning}
\tikzset{
  filling/.style={matrix of math nodes,
                  nodes={anchor=center, minimum size=1.5em},
                  column sep=.3em, row sep=.3em,
                  draw, inner sep=.3em},
  lpath/.style={line width=.8em, line cap=round, line join=round,
                draw=#1, opacity=.3},
  ldot/.style={fill=#1, opacity=.3},
}
\setlist[enumerate]{label=(\roman*)}
\usepackage[hidelinks]{hyperref}
\hypersetup{colorlinks=true, linkcolor=blue, citecolor=blue}
\usepackage{xcolor}

\newcommand{\lie}[1]{\mathfrak{#1}}
\newcommand{\R}{\mathbb R}
\newcommand{\Z}{\mathbb Z}
\newcommand{\N}{\mathbb N}
\newcommand{\C}{\mathbb C}
\DeclareMathOperator{\GL}{GL}
\DeclareMathOperator{\fold}{fold}
\DeclareMathOperator{\unfold}{unfold}
\renewcommand{\bar}{\overline}

\theoremstyle{plain}
\newtheorem{theorem}{Theorem}[section]
\numberwithin{equation}{section}
\theoremstyle{definition}
\newtheorem{lem}[theorem]{Lemma}
\newtheorem{cor}[theorem]{Corollary}
\newtheorem{prop}[theorem]{Proposition}
\newtheorem{exa}[theorem]{Example}
\newtheorem{defn}[theorem]{Definition}
\newtheorem{rem}[theorem]{Remark}

\title{FFLV polytopes are string polytopes}
\author{Ester Cleusters}
\address{University of Bonn, Germany}
\email{clstrs@math.uni-bonn.de}
\author{Ghislain Fourier}
\address{RWTH Aachen University, Germany}
\email{fourier@art.rwth-aachen.de}
\author{Felix Lerner}
\address{RWTH Aachen University, Germany}
\email{lerner@art.rwth-aachen.de}

\begin{document}
\begin{abstract}
We prove that FFLV polytopes of types $A_n$ and $C_n$ are affinely unimodularly equivalent to Demazure string polytopes of types $A_{2n-1}$ and $C_{2n-1}$, respectively.
The equivalence is given by affine unimodular maps whose linear part is independent of the highest weight and whose translation part depends linearly on it.
These maps are compatible with addition of dominant weights.
Consequently, the resulting Demazure string polytopes are lattice polytopes and satisfy the Minkowski property, both for polytopes and for lattice points.
\end{abstract}

\maketitle

\section{Introduction}

Let $\lie g$ be a simple complex Lie algebra and $V(\lambda)$ the finite-dimensional simple module of highest weight $\lambda$.
There are several constructions of monomial bases of $V(\lambda)$ which are parametrized by lattice points in convex polytopes.
The approach via essential monomials goes back to Vinberg; see, for example, Fang, the second author and Littelmann \cite{FaFL17}.
Among the best known polyhedral models are the string polytopes associated with reduced expressions of the longest Weyl group element, developed by Littelmann \cite{Lit98} and Berenstein and Zelevinsky \cite{BZ01}, and the FFLV polytopes in types $A$ and $C$, constructed by Feigin, the second author and Littelmann \cite{FFL11a,FFL11b}.

String polytopes arise from canonical bases, whereas FFLV polytopes arise from the PBW filtration.
In type $A$, Ardila, Bliem and Salazar \cite{ABS11} identified FFLV polytopes as marked chain polytopes and Gelfand--Tsetlin polytopes as marked order polytopes; see also the second author \cite{FourierMarked16} and Fang, the second author, Litza and Pegel \cite{FFLP20}.
In general, the FFLV polytope of $V(\lambda)$ is not unimodularly equivalent to a string polytope for the same module.
For instance, this already happens for $\lambda=\rho$ in type $A_4$, as shown by Bossinger, Lamboglia, Mincheva and Mohammadi \cite[Remark~6]{BLMM17}.
Relations between FFLV and string polytopes via iterated combinatorial mutations of Newton--Okounkov bodies were studied by Fujita and Higashitani \cite{FH21}.

On the other hand, the PBW filtration gives rise to the associated graded module $V^a(\lambda)$.
The corresponding degenerate flag varieties were introduced in type $A$ by Feigin \cite{Fei11} and in type $C$ by Feigin, Finkelberg and Littelmann \cite{FFiL14}.
Cerulli Irelli and Lanini \cite{CIL15} showed that these degenerate flag varieties are Schubert varieties in partial flag varieties of the same type and higher rank.
The representation-theoretic version was obtained by Cerulli Irelli, Lanini and Littelmann \cite{CILL16}: the module $V^a(\lambda)$ for a Lie algebra of rank $n$ is isomorphic to a Demazure module for a Lie algebra of the same type and rank $2n-1$.

We compare the FFLV and Demazure string parametrizations and prove that the corresponding polytopes are affinely unimodularly equivalent.

Let $X\in\{A,C\}$ and let
\[
\lambda=\sum_{i=1}^n a_i\omega_i
\]
be a dominant integral weight of type $X_n$.
Denote the fundamental weights of type $X_{2n-1}$ by $\widetilde\omega_j$ and set
\[
\widetilde\lambda=\sum_{i=1}^n a_i\widetilde\omega_{2i-1}.
\]
For the reduced expression $\underline w$ defined in Section~\ref{sec-main}, our main result is the following.

\begin{theorem}\label{thm:intro}
The FFLV polytope $P(\lambda)$ is affinely unimodularly equivalent to the Demazure string polytope $Q_{\underline w}(\widetilde\lambda)$.
The equivalences can be chosen compatibly with addition of dominant weights.
\end{theorem}

From the point of view of string polytopes, Theorem~\ref{thm:intro} singles out a family with stronger properties than those that hold in general.
For a fixed reduced expression one has
\[
Q_{\underline w}(\lambda)+Q_{\underline w}(\mu)
\subseteq Q_{\underline w}(\lambda+\mu),
\]
but equality need not hold.
Already in type $A$, the Minkowski property depends on the reduced expression \cite{BLMM17}, and string polytopes need not be lattice polytopes; see the discussion in Cho, Kim, Lee and Park \cite{CKLP21}.
For the word $\underline w$ above, Corollary~\ref{cor:minkowski} gives equality for all dominant weights in the submonoid
\[
\N\widetilde\omega_1+\N\widetilde\omega_3+\cdots+\N\widetilde\omega_{2n-1},
\]
also on lattice points.
Hence the corresponding Demazure string polytopes have the integer decomposition property, and their toric special fibers are projectively normal.

A related result holds in type $B$.
Makhlin \cite{Mak19} constructed FFLV-type polytopes associated with a weighted PBW filtration.
The second author \cite{Fou26B} shows that these polytopes are affinely unimodularly equivalent to higher-rank Demazure string polytopes.
The proof uses the type $C$ equivalence established here together with the $B/C$ similarity for string polytopes.

It remains open whether the isomorphism
\[
V^a(\lambda)\simeq V_w(\widetilde\lambda)
\]
can be obtained directly by relating suitable birational sequences on the two sides so that the corresponding essential exponents are identified (\cite{FaFL17}).

The paper is organized as follows.
In Section~\ref{sec-prel} we recall the necessary facts on monomial bases, string polytopes, FFLV polytopes and PBW degenerations.
In Section~\ref{sec-main} we state the main theorem and record its immediate consequences.
Section~\ref{sec-gen} presents the general setup, defines the affine maps and reduces the proof of the main theorem to Lemma~\ref{lem:mono}.
The type $A$ and type $C$ cases are treated in Sections~\ref{sec-a} and~\ref{sec-c}, respectively.\\

\noindent\textbf{Acknowledgments.}
The authors gratefully acknowledge financial support from the DFG---Project-ID 286237555---TRR 195.
The authors would like to thank Xin Fang, Martina Lanini and Christian Steinert for helpful discussions.
\\

\noindent\textbf{Use of AI.}
The main theorem was proved and written up by the authors without the use of AI.
Afterwards AI models ChatGPT~6 Astra and Claude Opus~5.5 were used to simplify and refine the exposition.

\section{Preliminaries}\label{sec-prel}

\subsection{Notation}
Throughout, we use $\N=\Z_{\geq0}$.
Let $\lie g=\lie n^+\oplus\lie h\oplus\lie n^-$ be a simple complex Lie algebra of rank $n$.
We denote its positive roots by $R^+$, its simple roots by $\alpha_1,\ldots,\alpha_n$, its fundamental weights by $\omega_1,\ldots,\omega_n$, and its dominant integral weights by $\Lambda^+$.
Let $W$ be the Weyl group, let $s_i$ be the simple reflection corresponding to $\alpha_i$, and let $w_0$ be the longest element.
For $\alpha\in R^+$, write $\alpha^\vee$ for its coroot.
If $\beta=\sum_{i=1}^n c_i\alpha_i\in R^+$, set
\[
\mathrm{ht}(\beta)=\sum_{i=1}^n c_i.
\]
We fix negative root vectors $f_\alpha$ and write $f_i=f_{\alpha_i}$.
For $\lambda\in\Lambda^+$, let $V(\lambda)$ be the simple highest weight module with highest weight vector $v_\lambda$.
For $w\in W$, choose a non-zero vector $v_{w(\lambda)}$ in the extremal weight space of weight $w(\lambda)$.
With $\lie b=\lie n^+\oplus\lie h$, the corresponding Demazure module is
\[
V_w(\lambda)=U(\lie b)v_{w(\lambda)}.
\]
For a rational polytope $P\subset\R^N$, write
\[
P^\Z=P\cap\Z^N.
\]
In type $A$, we normalize the simple negative root vectors by
\[
f_r=E_{r+1,r},
\]
where $E_{p,q}$ denotes the matrix unit in the ordered standard basis of the natural representation.

\subsection{Monomial bases}

We use essential monomials in the sense of \cite{FaFL17}.
Fix an ordered family $\mathbf f = (f_{\beta_1},\dots,f_{\beta_M})$ consisting of lowering operators whose monomials applied to $v_\lambda$ span $V(\lambda)$, and a monomial order $<$ on the exponent vectors.
An exponent $\mathbf m = (m_1,\dots,m_M) \in \N^M$ is essential for $V(\lambda)$ if
\[
\mathbf f^{\mathbf m}v_\lambda \notin \langle \mathbf f^\mathbf k v_\lambda \mid \mathbf k < \mathbf m \rangle_\mathbb C,
\]
where $\mathbf f^{\mathbf m}\coloneqq f_{\beta_1}^{m_1} \cdots f_{\beta_M}^{m_M}$.

The vectors corresponding to the essential exponents form a basis of $V(\lambda)$.

We define two orders on $\N^{M}$.
For distinct $\mathbf x,\mathbf y\in\N^{M}$, let $k$ be the smallest index with $x_k\neq y_k$.
We write
\[
\mathbf x <_{\mathrm{neglex}} \mathbf y
\quad:\Longleftrightarrow\quad
x_k>y_k
\]
and the height-weighted negative lexicographic order is defined as
\[
    \mathbf x <_{\mathrm{ht-neglex}} \mathbf y
\quad:\Longleftrightarrow\quad |\mathbf x| < |\mathbf y| \quad \text{or} \quad ( |\mathbf x| = |\mathbf y| \text{ and } \mathbf x <_{\mathrm{neglex}} \mathbf y),
\]
where
\[
    |\mathbf{x}| \coloneqq \sum_{j=1}^M x_j \cdot \mathrm{ht}(\beta_j).
\]
The height-weighted negative lexicographic order is a monomial order.

\medskip
\noindent\textit{String bases.}
Let $\underline{w_0}=s_{i_1}\cdots s_{i_M}$ be a reduced expression of the longest Weyl group element $w_0$.
The ordered family for the string bases is defined as
\begin{align*}
    \mathbf f \coloneqq (f_{\alpha_{i_1}},\dots,f_{\alpha_{i_M}}) = (f_{i_1},\dots,f_{i_M}).
\end{align*}

Let $C_{\underline{w_0}}\subset\R_{\geq0}^M$ be the rational string cone associated with $\underline{w_0}$, using the conventions of \cite{Lit98,BZ01}.
For the height-weighted negative lexicographic order, the essential exponents of $\mathbf f^{\mathbf{m}}v_\lambda$

are the string parameters; see \cite[Section~11, Theorem~1 and Corollary~5]{FaFL17}, together with \cite{Lit98,BZ01} for the string parametrization.
The string parameters are the lattice points of the string polytope
\[
Q_{\underline{w_0}}(\lambda)
= \left\{\mathbf t\in C_{\underline{w_0}}\;\middle|\;
t_k+\sum_{j>k}\langle\alpha_{i_j},\alpha_{i_k}^{\vee}\rangle t_j
\leq\langle\lambda,\alpha_{i_k}^{\vee}\rangle\quad(1\leq k\leq M)
\right\}.
\]
On a fixed weight space the total degree is fixed, so only the negative lexicographic part of the order is relevant below.
Let
\[
\underline w = s_{i_1}\cdots s_{i_\ell}
\]
be a reduced expression of $w\in W$ and extend it to a reduced expression
\[
\underline{w_0} = s_{i_1}\cdots s_{i_\ell}s_{i_{\ell+1}}\cdots s_{i_M}
\]
of $w_0$.
We identify the Demazure string polytope with the coordinate face
\[
Q_{\underline w}(\lambda)
=
\bigl\{(t_1,\ldots,t_\ell) \in \mathbb R^\ell_{\geq 0} \mid
(t_1,\ldots,t_\ell,0,\ldots,0)\in Q_{\underline{w_0}}(\lambda)\bigr\}.
\]
Its lattice points parametrize the string basis of $V_w(\lambda)$; see \cite{Cal02}.
The resulting parametrization is independent of the chosen reduced extension of $\underline w$ to $\underline{w_0}$.

The following consequence for the initial subword $\underline w$ is useful.

\begin{lem}\label{lem:truncated-essential}
Let $\underline w=s_{i_1}\cdots s_{i_\ell}$ be an initial subword of a reduced expression $\underline{w_0}$ as above, and set $\mathbf f=(f_{i_1},\ldots,f_{i_\ell})$.
An exponent $\mathbf x\in\N^\ell$ is essential for this ordered family acting on $v_\lambda$, with respect to the negative lexicographic order, if and only if
\[
\mathbf x\in Q_{\underline w}(\lambda)^\Z.
\]
\end{lem}

\begin{proof}
If $\mathbf x\in Q_{\underline w}(\lambda)^\Z$, then $(\mathbf x,0,\ldots,0)$ is essential for the full word $\underline{w_0}$.
If $\mathbf x$ were not essential for the truncated family, then
\[
f_{i_1}^{x_1}\cdots f_{i_\ell}^{x_\ell}v_\lambda \in \langle \mathbf f^\mathbf y v_\lambda \mid \mathbf y <_{\mathrm{neglex}} \mathbf x \rangle.
\]

Appending zeroes gives
\[
(\mathbf y,0,\ldots,0)<_{\mathrm{neglex}}(\mathbf x,0,\ldots,0),
\]
contradicting essentiality for $\underline{w_0}$.
Thus
\[
Q_{\underline w}(\lambda)^\Z
\]
is contained in the set of essential exponents for the truncated family.

Conversely, the monomials attached to the reduced word $\underline w$ span the Demazure module $V_w(\lambda)$ by the usual Demazure recursion.
Hence the essential exponents for the truncated family form a basis of $V_w(\lambda)$ and their number is $\dim V_w(\lambda)$.
The same is true for $Q_{\underline w}(\lambda)^\Z$ by the Demazure string parametrization.
The preceding inclusion is therefore an equality.
\end{proof}

\begin{rem}\label{rem:one-dimensional}
On a one-dimensional weight space, a non-zero monomial is essential if and only if its exponent is minimal among the non-zero monomials of that weight.
\end{rem}

For fixed $\underline w$,
\[
Q_{\underline w}(\lambda)+Q_{\underline w}(\mu)
\subseteq
Q_{\underline w}(\lambda+\mu) \quad \text{and} \quad Q_{\underline w}(m\lambda)=mQ_{\underline w}(\lambda)
\]

for $\lambda,\mu\in \Lambda^+$ and $m\in\N$.

\medskip
\noindent\textit{FFLV bases.}
In types $A$ and $C$, the FFLV polytopes were constructed in \cite{FFL11a,FFL11b}.
We recall the construction.
In both types, we write
\[
\alpha_{a,b}=\alpha_a+\cdots+\alpha_b,
\qquad
1\leq a\leq b\leq n,
\]
and in type $C_n$ additionally
\[
\alpha_{a,\bar b}
=
\alpha_a+\cdots+\alpha_{b-1}
+2\alpha_b+\cdots+2\alpha_{n-1}+\alpha_n,
\qquad
1\leq a\leq b\leq n-1.
\]
The barred symbols are formal labels.
Second indices are compared in the order
\[
1<2<\cdots<n=\bar n<\overline{n-1}<\cdots<\bar1.
\]

A Dyck path $\mathbf p = (p(0),\dots,p(\ell))$ is a sequence of positive roots starting at a simple root $\alpha_i$, ending at a simple root $\alpha_j$ or, in type $C$, at a root $\alpha_{j,\bar j}$ for $j \geq i$ such that for each $0 \leq s \leq \ell-1$:
\begin{align*}
    \text{if } p(s) = \alpha_{p,q} \quad \text{then} \quad  p(s+1) \in \{\alpha_{p,q+1},\alpha_{p+1,q}\}
\end{align*}
where $n+1 \coloneqq \bar{n-1}$ and $\bar{k}+1 \coloneqq \bar{k-1}$.
We denote the set of Dyck paths in type $X_n$ (for $X \in \{A,C\}$) by $\mathbb D(X_n)$.
Fix a dominant integral weight $\lambda = \sum_{i=1}^n a_i \omega_i \in \Lambda^+$.
The $\lambda$-weight of a Dyck path $\mathbf p = (p(0),\dots,p(\ell)) \in \mathbb D(X_n)$ is defined as
\begin{align*}
    \mathrm{wt}_\lambda(\mathbf p) \coloneqq
    \begin{cases}
        \sum\limits_{i \leq k \leq j} a_k & \text{if } p(0) = \alpha_i, p(\ell) = \alpha_j \\
        \sum\limits_{i \leq k \leq n} a_k & \text{if } p(0) = \alpha_i, p(\ell) = \alpha_{j,\bar j}.
    \end{cases}
\end{align*}

The FFLV polytope $P(\lambda)$ for type $X_n$ of highest weight $\lambda$ is defined as
\begin{align*}
    P(\lambda) = \left\{(x_\alpha)_{\alpha \in R^+} \in \R_{\geq 0}^{R^+} \mid
    \forall \mathbf p \in \mathbb D(X_n): \sum_{s \in \mathbf p} x_s \leq \mathrm{wt}_{\lambda}(\mathbf p)
    \right\},
\end{align*}
where the sum goes over all steps $s = p(0),\dots,p(\ell)$ of the Dyck path $\mathbf p$.

The lattice points of $P(\lambda)$ parametrize an essential monomial basis of $V(\lambda)$, and $P(\lambda)$ is a lattice polytope.
Furthermore, the FFLV polytope satisfies the Minkowski property
\[
P(\lambda)=\sum_{i=1}^n a_iP(\omega_i).
\]
Moreover, for $\lambda,\mu\in \Lambda^+$,
\[
P(\lambda)^\Z+P(\mu)^\Z=P(\lambda+\mu)^\Z;
\]
see \cite{FFL11a,FFL11b}.
\subsection{PBW degeneration}

The PBW filtration on $U(\lie n^-)$ induces a filtration on $V(\lambda)$.
We write
\[
V^a(\lambda)=\operatorname{gr}V(\lambda).
\]
In types $A$ and $C$, the FFLV basis induces a basis of $V^a(\lambda)$.

By \cite{CILL16}, the modules considered here admit a realization
\[
V^a(\lambda)\simeq V_w(\widetilde\lambda)
\]
as Demazure modules of the same type and higher rank, with the abelian Lie algebra action of \cite{CILL16}.
The weight $\widetilde\lambda$ and a reduced expression $\underline{w}$ of the element $w$ will be made explicit in the next section.
\section{Main theorem}\label{sec-main}

Let $X\in\{A,C\}$ and
\[
\lambda=\sum_{i=1}^n a_i\omega_i\in \Lambda^+.
\]
We denote the simple roots and fundamental weights of type $X_{2n-1}$ by $\widetilde\alpha_1,\ldots,\widetilde\alpha_{2n-1}$ and $\widetilde\omega_1,\ldots,\widetilde\omega_{2n-1}$, respectively.
In the reduced word $\underline w$ and its associated ordered monomials, $s_j$ and $f_j$ denote the simple reflection and negative simple-root vector of this higher-rank Lie algebra.
We set
\[
\widetilde\lambda
=\sum_{i=1}^n a_i\widetilde\omega_{2i-1}.
\]
In type $A_n$, we use the ordered standard basis $e_1,\ldots,e_{n+1}$ of the natural representation.
Let $N$ be the number of positive roots for type $X_n$, that is $N=|R^+(X_n)|$.
For $1\leq j\leq n$, set
\[
\sigma_j=s_js_{j+1}\cdots s_{2j-1},
\]
and, in type $C$, for $n+1\leq j\leq2n-1$,
\[
\tau_j=s_js_{j+1}\cdots s_{2n-1}.
\]
We consider the reduced expression
\[
\underline w=
\begin{cases}
\sigma_n\sigma_{n-1}\cdots\sigma_1,
&X=A,\\[1mm]
\tau_{2n-1}\tau_{2n-2}\cdots\tau_{n+1}
\sigma_n\sigma_{n-1}\cdots\sigma_1,
&X=C.
\end{cases}
\]
These are, up to renumbering, the elements of \cite[Definitions~2.2 and~5.1]{CILL16}.
The lengths are
\[
\ell(w)=
\begin{cases}
n(n+1)/2,&X=A,\\
n^2,&X=C.
\end{cases}
\]
In both cases $\ell(w)=N$.
The affine maps are defined in Section~\ref{sec-gen}.
We use the coordinate identification of Definition~\ref{ordering} to regard their images as string-coordinate vectors.
\begin{theorem}\label{thm:main}
For every $\lambda\in \Lambda^+$ there exists an affine unimodular map
\[
T_{X,n,\lambda}:\R^N\longrightarrow\R^N,
\qquad
T_{X,n,\lambda}(\mathbf x)
=
\Phi_{X,n}(\mathbf x) + \mathbf t_{X,n,\lambda},
\]
such that
\[
T_{X,n,\lambda}\bigl(P(\lambda)\bigr)
=
Q_{\underline w}(\widetilde\lambda),
\]
the linear part $\Phi_{X,n}\in\GL_N(\Z)$ is independent of $\lambda$, and $\mathbf t_{X,n,\lambda}\in\Z^N$ depends linearly on $\lambda$.
\end{theorem}

The family $(T_{X,n,\lambda})_{\lambda\in\Lambda^+}$ from Theorem~\ref{thm:main} satisfies
\[
T_{X,n,\lambda+\mu}(\mathbf x+\mathbf y)
=
T_{X,n,\lambda}(\mathbf x)+T_{X,n,\mu}(\mathbf y)
\]
because $\mathbf t_{X,n,\lambda}$ is linear in $\lambda$.
\begin{cor}\label{cor:minkowski}
For all $\lambda,\mu\in \Lambda^+$,
\[
Q_{\underline w}(\widetilde\lambda)
+
Q_{\underline w}(\widetilde\mu)
=
Q_{\underline w}(\widetilde{\lambda+\mu})
\]
and
\[
Q_{\underline w}(\widetilde\lambda)^\Z
+
Q_{\underline w}(\widetilde\mu)^\Z
=
Q_{\underline w}(\widetilde{\lambda+\mu})^\Z.
\]
In particular, $Q_{\underline w}(\widetilde\lambda)$ is a lattice polytope.
\end{cor}

\begin{proof}
The first equality follows from Theorem~\ref{thm:main} and the Minkowski property of FFLV polytopes.
For lattice points, use
\[
P(\lambda)^\Z+P(\mu)^\Z=P(\lambda+\mu)^\Z
\]
and that for a family of maps $T_{X,n,\lambda}$ from Theorem ~\ref{thm:main}
\[
T_{X,n,\lambda+\mu}(\mathbf x+\mathbf y)
=
T_{X,n,\lambda}(\mathbf x)+T_{X,n,\mu}(\mathbf y).
\]
The last assertion of the corollary follows from the unimodularity of $T_{X,n,\lambda}$.
\end{proof}

In particular,
\[
Q_{\underline w}(\widetilde\lambda)^\Z
=
\sum_{i=1}^n a_iQ_{\underline w}(\widetilde\omega_{2i-1})^\Z.
\]
Here $a_iS$ denotes the $a_i$-fold Minkowski sum of $S$.
Taking repeated sums of $Q_{\underline w}(\widetilde\lambda)^\Z$ gives the integer decomposition property, i.e. for all $k \in \N$,
\[
    (kQ_{\underline w}(\widetilde\lambda))^\Z = Q_{\underline w}(k\widetilde\lambda)^\Z = k Q_{\underline w}(\widetilde\lambda)^\Z.
\]

\begin{rem}
The toric degeneration defined by the FFLV polytope is known from \cite{FFL17}.
Theorem~\ref{thm:main} identifies the same polytope, up to affine unimodular equivalence, with a string polytope for the higher-rank Demazure realization of \cite{CIL15,CILL16}.
\end{rem}

\section{General setup}\label{sec-gen}

Fix $X\in\{A,C\}$.
We set
\[
H(A_n)=\{(a,b)\mid1\leq a\leq b\leq n\}
\]
and
\[
H(C_n)=H(A_n)\cup
\{(a,\bar b)\mid1\leq a\leq b\leq n-1\}.
\]
For $\gamma\in H(X_n)$, let $e_\gamma$ denote the corresponding standard basis vector of $\R^{H(X_n)}$.
In type $C$, we use
\[
e_{a,\bar n}=e_{a,n},
\qquad
x_{a,\bar n}=x_{a,n};
\]
thus $(a,\bar n)$ does not index an additional coordinate.

\begin{defn}\label{ordering}
We order $H(X_n)$ by
\[
(a,b)\preceq(c,d)
\quad:\Longleftrightarrow\quad
b<d
\quad\text{or}\quad
(b=d\text{ and }a\geq c),
\]
and write $\prec$ for the associated strict order.
Let $\gamma_1\prec\cdots\prec\gamma_N$ be the resulting ordering.
We identify $\R^{H(X_n)}$ with $\R^N$ by
\[
\sum_{\gamma\in H(X_n)}x_\gamma e_\gamma
\longmapsto
(x_{\gamma_N},\ldots,x_{\gamma_1}).
\]
In particular, the $k$-th string coordinate is $x_{\gamma_{N+1-k}}$.
\end{defn}

For unbarred and barred indices, respectively, set
\[
r(a,j)=a+j-1,
\qquad
r(a,\bar j)=2n-j+a-1.
\]
The associated ordered family of lowering operators is
\[
\underline{\mathbf f}
=
(f_{r(\gamma_N)},\ldots,f_{r(\gamma_1)}).
\]
For $\mathbf x\in\N^{H(X_n)}$, we write
\[
\underline{\mathbf f}^{\mathbf x}
=
f_{r(\gamma_N)}^{x_{\gamma_N}}
\cdots
f_{r(\gamma_1)}^{x_{\gamma_1}}.
\]
For distinct $\mathbf x,\mathbf y\in\N^{H(X_n)}$, let $\gamma$ be the largest index with respect to $\preceq$ for which $x_\gamma\neq y_\gamma$.
Under the coordinate identification of Definition~\ref{ordering}, we have
\[
\mathbf x<_{\mathrm{neglex}}\mathbf y
\quad\Longleftrightarrow\quad
x_\gamma>y_\gamma.
\]

\begin{exa}
The ordering on $H(C_3)$ is
\[
\begin{aligned}
(1,\bar1)
&\succ(1,\bar2)\succ(2,\bar2)
\succ(1,3)\succ(2,3)\succ(3,3)\\
&\succ(1,2)\succ(2,2)\succ(1,1),
\end{aligned}
\]
and the corresponding ordered family of lowering operators is
\[
(f_5,f_4,f_5,f_3,f_4,f_5,f_2,f_3,f_1).
\]
\end{exa}

\begin{lem}\label{lem:word-order}
The ordered family $\underline{\mathbf f}$ coincides with the family of lowering operators for the reduced expression $\underline w$ of Section~\ref{sec-main}.
\end{lem}

\begin{proof}
In type $A$, the column with second coordinate $j$ contributes the block
\[
f_{r(1,j)}f_{r(2,j)}\cdots f_{r(j,j)}
=
f_jf_{j+1}\cdots f_{2j-1},
\]
corresponding to $\sigma_j$.
The columns are read from $j=n$ to $j=1$, giving $\sigma_n,\ldots,\sigma_1$.
In type $C$, the unbarred columns give the same blocks.
The barred column $\bar j$ contributes
\[
f_{r(1,\bar j)}f_{r(2,\bar j)}\cdots f_{r(j,\bar j)}
=
f_{2n-j}f_{2n-j+1}\cdots f_{2n-1},
\]
corresponding to $\tau_{2n-j}$.
The barred columns are read as $\bar1,\bar2,\ldots,\overline{n-1}$, giving $\tau_{2n-1},\ldots,\tau_{n+1}$.
Thus the resulting word is $\underline w$.
\end{proof}

We shall also use $f_\ell f_j=f_jf_\ell$ for $|\ell-j|>1$, since the corresponding Dynkin nodes are not adjacent in both types $A$ and $C$.

We now define the linear maps $\Phi_{X,n}:\R^{H(X_n)}\longrightarrow\R^{H(X_n)}$.

\begin{defn}\label{def cX}
In type $A_n$, define $\Phi_{A,n}$ by the $\R$-linear extension of
\[
\Phi_{A,n}(e_{a,b})
=
-\sum_{c=b}^n e_{a,c}
-\sum_{c<a}e_{c,b}.
\]
In type $C_n$, define $\Phi_{C,n}$ by the $\R$-linear extension of
\[
\Phi_{C,n}(e_{a,b})
=
-\begin{cases}
e_{a,\bar a}
+2\displaystyle\sum_{c<a}e_{c,b},
&b=\bar a,\\[3mm]
\displaystyle\sum_{\bar a\geq c\geq b}e_{a,c}
+\displaystyle\sum_{c<a}
\bigl(e_{c,b}+e_{c,\bar a}\bigr),
&b\neq\bar a.
\end{cases}
\]
\end{defn}

\begin{prop}\label{lem:unimod}
For $X\in\{A,C\}$, we have $\Phi_{X,n}\in\GL(\Z^{H(X_n)})$.
\end{prop}

The two cases are proved separately in Sections~\ref{sec-a} and~\ref{sec-c}.

\begin{defn}\label{tau}
For $1\leq i\leq n$, set
\[
\mathbf t_{X,n,\omega_i}
=
\sum_{(a,b)\in H(X_n)}
\kappa^i_{a,b}(X)e_{a,b},
\]
where
\[
\kappa^i_{a,b}(A)
=
\begin{cases}
1,&a\leq i\leq b,\\
0,&\text{otherwise},
\end{cases}
\]
and
\[
\kappa^i_{a,b}(C)
=
\begin{cases}
1,&i\leq b<\bar i\text{ and }a\leq i,\\
1,&b=\bar a\text{ and }a\leq i,\\
2,&\bar i\leq b<\bar a,\\
0,&\text{otherwise}.
\end{cases}
\]
For $\lambda=\sum_{i=1}^n a_i\omega_i$, define
\[
\mathbf t_{X,n,\lambda}
=
\sum_{i=1}^n a_i\mathbf t_{X,n,\omega_i}.
\]
\end{defn}

For the rest of this section, write $T_\lambda=T_{X,n,\lambda}$, where
\[
T_\lambda(\mathbf x)
=
\Phi_{X,n}(\mathbf x)+\mathbf t_{X,n,\lambda}.
\]
It remains to establish the following equality for fundamental weights.

\begin{lem}\label{lem:mono}
For $X\in\{A,C\}$ and $1\leq i\leq n$,
\[
T_{\omega_i}\bigl(P(\omega_i)^\Z\bigr)
=
Q_{\underline w}(\widetilde\omega_{2i-1})^\Z.
\]
\end{lem}

Lemma~\ref{lem:mono} is proved separately in Sections~\ref{sec-a} and~\ref{sec-c}.
We first show how it implies Theorem~\ref{thm:main}.

\begin{proof}[Proof of Theorem~\ref{thm:main}]
Since $P(\omega_i)$ is a lattice polytope, Lemma~\ref{lem:mono} gives
\[
T_{\omega_i}\bigl(P(\omega_i)\bigr)
\subseteq
Q_{\underline w}(\widetilde\omega_{2i-1}).
\]
Fix $\lambda=\sum_{i=1}^n a_i\omega_i\in\Lambda^+$.
Using the Minkowski property of the FFLV polytopes and the corresponding inclusion for string polytopes, we obtain
\begin{align*}
T_\lambda\bigl(P(\lambda)\bigr)
&=
\sum_{i=1}^n a_iT_{\omega_i}\bigl(P(\omega_i)\bigr)\\
&\subseteq
\sum_{i=1}^n a_iQ_{\underline w}(\widetilde\omega_{2i-1})\\
&\subseteq
Q_{\underline w}(\widetilde\lambda).
\end{align*}

For every $m\geq1$, we have
\[
T_{m\lambda}\bigl(P(m\lambda)\bigr)
=
mT_\lambda\bigl(P(\lambda)\bigr).
\]
Since $T_{m\lambda}$ is integral unimodular by Proposition~\ref{lem:unimod} and Definition~\ref{tau}, it follows that
\[
\#\bigl(mT_\lambda(P(\lambda))\bigr)^\Z
=
\#P(m\lambda)^\Z
=
\dim V(m\lambda).
\]
On the other hand, the equality $Q_{\underline w}(m\widetilde\lambda)=mQ_{\underline w}(\widetilde\lambda)$, the string parametrization and the Demazure realization give
\[
\#\bigl(mQ_{\underline w}(\widetilde\lambda)\bigr)^\Z
=
\dim V_w(m\widetilde\lambda)
=
\dim V^a(m\lambda)
=
\dim V(m\lambda).
\]
Thus the inclusion and equality of cardinalities imply, for every $m\geq1$,
\begin{equation}\label{Stern}
\bigl(mT_\lambda(P(\lambda))\bigr)^\Z
=
\bigl(mQ_{\underline w}(\widetilde\lambda)\bigr)^\Z.
\end{equation}

Assume for the sake of contradiction that $T_\lambda(P(\lambda))\subsetneq Q_{\underline w}(\widetilde\lambda)$.
Since $Q_{\underline w}(\widetilde\lambda)$ is rational, we may choose a point
\[
\mathbf q\in
Q_{\underline w}(\widetilde\lambda)
\setminus T_\lambda(P(\lambda))
\]
with rational coordinates.
Clearing denominators gives an integer $m\geq1$ such that $m\mathbf q\in\Z^N$.
Then
\[
m\mathbf q\in
\bigl(mQ_{\underline w}(\widetilde\lambda)\bigr)^\Z
\qquad\text{and}\qquad
m\mathbf q\notin
\bigl(mT_\lambda(P(\lambda))\bigr)^\Z,
\]
contradicting \eqref{Stern}.
Hence $T_\lambda(P(\lambda))=Q_{\underline w}(\widetilde\lambda)$.
The remaining assertions follow from Proposition~\ref{lem:unimod} and the linear dependence of $\mathbf t_{X,n,\lambda}$ on $\lambda$ in Definition~\ref{tau}.
\end{proof}

\section{Type \texorpdfstring{$A$}{A}}\label{sec-a}

\begin{proof}[Proof of Proposition~\ref{lem:unimod} in type $A$]
Consider the ordered basis
\[
\mathcal B_n=
(e_{1,n},\ldots,e_{n,n},
e_{1,n-1},\ldots,e_{n-1,n-1},
\ldots,e_{1,1}).
\]
For $1\leq a\leq b\leq n$,
\[
-\Phi_{A,n}(e_{a,b})
=
\sum_{c=b}^n e_{a,c}
+
\sum_{c<a}e_{c,b}.
\]
All terms in the above sum different from $e_{a,b}$ precede $e_{a,b}$ in $\mathcal B_n$.
Hence with respect to the basis $\mathcal B_n$, the map $-\Phi_{A,n}$ is an upper triangular matrix with diagonal entries equal to $1$, and therefore $\Phi_{A,n} \in \GL(\Z^{H(A_n)})$.
\end{proof}

For $1\leq i\leq n$, set
\[
R_i=\{(a,b)\in H(A_n)\mid a\leq i\leq b\}.
\]

The set $R_i$ is an $i\times(n-i+1)$ rectangle.
Note that the position $(a,b)$ corresponds to the operator $f_{a+b-1}$.
The monomial associated to $\mathbf x \in \N^{H(A_n)}$ is written column-wise from right to left and from top to bottom; its action is therefore read from left to right and from bottom to top.
We illustrate this with an example.

\begin{exa}
    Let $n=4$, $i=2$, and
    \[
    \mathbf x=e_{2,2}+e_{1,2}+e_{1,3}+e_{2,3} \in \N^{H(A_4)}.
    \]
    The support of $\mathbf x$ is contained in $R_2$.
    We may visualize $\mathbf x$ as
   \begin{table}[ht]
       \centering
       \begin{tabular}{c|ccc|}
       \cline{2-4}
          0 & 1 & 1 & 0 \\
            & 1 & 1 & 0 \\
          \cline{2-4}
          \multicolumn{1}{c}{} & & 0 & \multicolumn{1}{c}{0} \\
          \multicolumn{1}{c}{} & &   & \multicolumn{1}{c}{0}
        \end{tabular}.
   \end{table} \\
   The entries in the rectangle are the coefficients $x_{a,b}$.
   The corresponding operator is $\underline{\mathbf f}^{\mathbf x} = f_3f_4f_2f_3$ and we have
   \begin{align*}
       \underline{\mathbf f}^{\mathbf x}(e_1 \wedge e_2 \wedge e_3) = e_1 \wedge e_4 \wedge e_5.
   \end{align*}
\end{exa}

\begin{prop}\label{prop:restrA}
Let $v_i=e_1\wedge\cdots\wedge e_{2i-1} \in\Lambda^{2i-1}\C^{2n}$.
If
\[
\mathbf x=\sum_{(a,b)\in H(A_n)}x_{a,b}e_{a,b}\in\N^{H(A_n)}
\]
satisfies $\underline{\mathbf f}^\mathbf x v_i\neq0$, then $\mathbf x$ is a $0$--$1$ vector supported on $R_i$.
\end{prop}

\begin{proof}
Since $f_r^2$ acts as zero on $V(\widetilde\omega_{2i-1})$ for every simple-root index $r$, we have $x_{a,b}\leq1$ for all $(a,b) \in H(A_n)$.

Suppose that some $x_{a,b}\neq0$ with $b<i$, and choose the first such factor in action order.
Since $a\leq b\leq i-1$,
\[
a+b\leq2i-2.
\]
All preceding factors with second coordinate less than $i$ have exponent zero.
Thus both $e_{a+b-1}$ and $e_{a+b}$ occur in the wedge when
\[
f_{a+b-1}=E_{a+b,a+b-1}
\]
acts, and the wedge is annihilated.
Hence $x_{a,b}=0$ for $b<i$.

Before column $i$ acts, the largest occupied index is $2i-1$.
Within each column, the operators act with decreasing indices, so each wedge factor moves at most once.
Consequently, before column $b$, every occupied index is at most $i+b-1$.
For $a>i$, the source $a+b-1\geq i+b$ is therefore empty when the corresponding factor acts, and hence $x_{a,b}=0$.
\end{proof}

\begin{lem}\label{lem:path-family}
Let $\mathbf x$ be a $0$--$1$ filling of $R_i$ such that $\underline {\mathbf f}^\mathbf x v_i\neq0$.
For $q\in\Z$, set
\[
Z_q(\mathbf x)\coloneqq \{a \in \{1,\dots,n\} \mid (a,q-a+1)\in R_i,\ x_{a,q-a+1}=0\}.
\]
If we denote by $a_1 < \dots < a_r$ the elements of $Z_q(\mathbf x)$ in strictly increasing order and by $a'_1<\cdots<a'_{r'}$ the elements of $Z_{q+1}(\mathbf x)$ in strictly increasing order, then $|r-r'| \leq 1$.

Furthermore, we define the following matching
\begin{enumerate}
\item if $r'=r$, match $a_k$ with $a'_k$;
\item if $r'=r+1$, then $a'_1=1$; match $a_k$ with $a'_{k+1}$;
\item if $r'=r-1$, then $a_1=q-n+1$; match $a_{k+1}$ with $a'_k$.
\end{enumerate}
In every matched pair one has $a'\in\{a,a+1\}$.
Hence the zero set of $\mathbf x$ is a disjoint union of non-intersecting lattice paths, each starting in row $1$ and ending in column $n$.
\end{lem}

\begin{exa}
    In type $A_6$, let $i=3$ and consider $\mathbf x = e_{3,3} + e_{1,6}$.
    Then $\underline{\mathbf f}^{\mathbf x} = f_6 f_5$ and $\underline{\mathbf f}^{\mathbf x}(e_1 \wedge \dots \wedge e_5) = e_1 \wedge e_2 \wedge e_3 \wedge e_4 \wedge e_7 \neq 0$.
    In the following visual representation of $\mathbf x$, we only display the $R_3$ part of $H(A_6)$ and the lattice paths from the above lemma are marked:
    \begin{table}[ht]
          \centering
          \begin{tikzpicture}[
              lpath/.style={line width=.8em, line cap=round, line join=round, draw=#1, opacity=.3},]
            \matrix (M) [matrix of math nodes,
                         nodes={anchor=center, minimum size=1.5em},
                         column sep=.3em, row sep=.3em,
                         draw, inner sep=.3em]
            {
              0 & 0 & 0 & 1 \\
              0 & 0 & 0 & 0 \\
              1 & 0 & 0 & 0 \\
            };
            \begin{scope}[on background layer]
              \draw[lpath=green!60!black]
                (M-1-1.center) -- (M-2-1.center) -- (M-2-2.center)
                -- (M-3-2.center) -- (M-3-4.center);
              \draw[lpath=teal!80!green]
                (M-1-2.center) -- (M-1-3.center) -- (M-2-3.center) -- (M-2-4.center);
            \end{scope}
        \end{tikzpicture}
    \end{table}
\end{exa}

\begin{proof}[Proof of Lemma \ref{lem:path-family}]
    We read the indices of $R_i$ with respect to the order $\prec$, i.e. column by column from left to right and inside each column from bottom to top.
    The points of the form $(a,q-a+1) \in R_i$ each correspond to the operator $f_q$.
    Let $c_1,\dots,c_m$ be the indices on the anti-diagonals $q$ and $q+1$ in action order, i.e. $c_1 \prec \cdots \prec c_m$.
    If both anti-diagonals are empty, then $Z_q(\mathbf x)=Z_{q+1}(\mathbf x)=\varnothing$ and there is nothing to match.
    We may therefore assume $m\geq1$.
    We call an index a \textit{$q$-index} if it is an index on the $q$-anti-diagonal.
    By definition of $\prec$, the indices $c_1,\dots,c_m$ alternate between $q$-indices and $(q+1)$-indices.
    The first index $c_1$ is a $(q+1)$-index if and only if $q \leq 2i-2$ and the last index $c_m$ is a $(q+1)$-index if and only if $q < n$.
    Furthermore, we have
    \begin{align}\label{dagger}
        c_m = \begin{cases}
            (1,q+1) & \text{if } q < n \\
            (q-n+1,n) & \text{if } q \geq n.
        \end{cases}
    \end{align}
    We call an index $c$ \textit{active} if $x_c = 1$.
    If there are no active indices, pair consecutive indices in the alternating list $c_1,\dots,c_m$, leaving only $c_m$ unmatched when $m$ is odd.
    Each pair consists of a $q$-index and a $(q+1)$-index in adjacent cells.
    Now suppose that at least one index is active.
    The active indices also alternate between $q$-indices and $(q+1)$-indices since $\underline {\mathbf f}^\mathbf x v_i \neq 0$.
    This implies that between two active indices there is an even number of inactive indices.
    The first active index is a $(q+1)$-index if and only if $q+1 \leq 2i-1$ which is the case if and only if $c_1$ is a $(q+1)$-index.
    Thus, before the first active index, there is also an even number of inactive indices.
    Only after the last active index, there might be one more inactive $q$ or $q+1$ index, which implies $|r-r'| \leq 1$ and gives an adjacent (i.e. $a' \in \{a,a+1\}$) matching of all inactive indices except possibly $c_m$.
    
    In either case, if the last index $c_m$ is matched, then we have $r = r'$.
    To prove the assertions in (ii) and (iii), let us assume that $r' \neq r$.
    Then $c_m$ is not matched.
    \begin{itemize}[leftmargin=0.2in]
        \item If $r' = r-1$, then $c_m$ is a $q$-index, which gives $q \geq n$ and $a_1 = q-n+1$ by ~\eqref{dagger}.
        \item If $r' = r + 1$, then $c_m$ is a $(q+1)$-index, which gives $q < n$ and $a_1' = 1$ by \eqref{dagger}.
    \end{itemize}
    
    This also shows that the matching constructed in the proof is the matching described in the statement of the lemma.

    Consecutively combining matched pairs across the anti-diagonals provides disjoint lattice paths partitioning the zero set of $\mathbf x$.
    Such a path starts at a previously unmatched $a_1'$ and thus always starts in the first row by (ii).
    The endpoints of the previously described lattice paths are the unmatched $a_1$'s and thus always lie in the last column by (iii).
\end{proof}

\begin{lem}\label{lem:corner-swap}
Let $\mathbf x,\mathbf x'$ be $0$--$1$ fillings of $R_i$ which agree outside the two cells $(a+1,b)$ and $(a,b+1)$.
Assume that $(a+1,b)$ belongs to $R_i$,
\[
x_{a,b}=x_{a+1,b+1}=0,
\]
and that
\[
(x_{a+1,b},x_{a,b+1})=(1,0),
\qquad
(x'_{a+1,b},x'_{a,b+1})=(0,1).
\]
Then
\[
\underline {\mathbf f}^\mathbf x=\underline {\mathbf f}^{\mathbf x'}.
\]
\end{lem}

\begin{proof}
The two cells $(a+1,b)$ and $(a,b+1)$ correspond to the same simple-root operator $f_{a+b}$.
Passing this factor from one of these positions to the other crosses only factors which commute with $f_{a+b}$.
Indeed, the only possible neighbouring simple-root indices are $a+b-1$ and $a+b+1$, and the corresponding corner cells $(a,b)$ and $(a+1,b+1)$ have exponent zero.
Hence the two ordered monomials are equal in $U(\lie n^-_{A_{2n-1}})$.
\end{proof}

For $(\ell,j)\in R_i$, the hook from $(1,j)$ to $(\ell,n)$ is defined as
\[
\eta_{\ell,j}=\sum_{b=j}^n e_{\ell,b}+\sum_{a<\ell}e_{a,j}.
\]
Thus $\Phi_{A,n}(e_{\ell,j})=-\eta_{\ell,j}$ and $T_{A,n,\omega_i}(e_{\ell,j}) = \mathbf t_{A,n,\omega_i} - \eta_{\ell,j}$.

\begin{exa}
For $n=4$ and $i = 2$, let $\mathbf x=e_{2,3}+e_{1,4}$ and write $T\coloneqq T_{A,4,\omega_2}$.
Then
\[
\mathbf \eta_{2,3}=e_{1,3}+e_{2,3}+e_{2,4},
\qquad
\eta_{1,4}=e_{1,4}
\]
and therefore
\[
T(\mathbf x)= \mathbf t_{A,4,\omega_2}-\eta_{2,3}-\eta_{1,4}
=e_{1,2}+e_{2,2}.
\]

We may visualize this as follows:
\[
\mathbf x=
\vcenter{\hbox{\begin{tikzpicture}
  \matrix (M) [filling] { 0 & 0 & 1 \\ 0 & 1 & 0 \\ };
  \begin{scope}[on background layer]
    \fill[ldot=green!60!black] (M-2-2.center) circle (.4em);
    \fill[ldot=teal!80!green] (M-1-3.center) circle (.4em);
  \end{scope}
\end{tikzpicture}}}
\mapsto
\vcenter{\hbox{\begin{tikzpicture}
  \matrix (M) [filling] { 1 & 1 & 1 \\ 1 & 1 & 1 \\ };
\end{tikzpicture}}}
-
\vcenter{\hbox{\begin{tikzpicture}
  \matrix (M) [filling] { 0 & 1 & 1 \\ 0 & 1 & 1 \\ };
  \begin{scope}[on background layer]
    \draw[lpath=green!60!black] (M-1-2.center) -- (M-2-2.center) -- (M-2-3.center);
    \fill[ldot=teal!80!green] (M-1-3.center) circle (.4em);
  \end{scope}
\end{tikzpicture}}}
=
\vcenter{\hbox{\begin{tikzpicture}
  \matrix (M) [filling] { 1 & 0 & 0 \\ 1 & 0 & 0 \\ };
\end{tikzpicture}}}
=T(\mathbf x)
\]
    We have
    \[
    \underline {\mathbf f}^{T(\mathbf x)}=f_2f_3,
    \qquad \text{and} \qquad
    f_2f_3(e_1\wedge e_2\wedge e_3)
    =e_1\wedge e_3\wedge e_4.
    \]
\end{exa}

\begin{lem}\label{lem:surjective-lattice-points}
    The map $T_{A,n,\omega_i}$ from Section \ref{sec-gen} satisfies $Q_{\underline{w}}(\tilde{\omega}_{2i-1})^\Z \subseteq T_{A,n,\omega_i}(P(\omega_i)^\Z)$.
\end{lem}

\begin{proof}
    Let $\mathbf x \in Q_{\underline{w}}(\tilde{\omega}_{2i-1})^\Z$.
    By Lemma \ref{lem:truncated-essential}, $\mathbf x$ is essential.
    Since $V(\tilde{\omega}_{2i-1})$ is minuscule, Remark \ref{rem:one-dimensional} implies that $\mathbf x$ is minimal with respect to neglex among the non-zero monomials of the same weight.
    Since $\underline{\mathbf f}^{\mathbf x} v_i \neq 0$, Proposition \ref{prop:restrA} implies that $\mathbf x$ is a $0$--$1$ filling of $R_i$.
    Denote the collection of lattice paths that partition the zero set of $\mathbf x$, given in Lemma \ref{lem:path-family}, by $\Pi$.
    We show that every lattice path of $\Pi$ has at most one corner and if it has a corner, it is a south-west corner (i.e. a down step followed by a right step).

    Among all north-east corners $(a,b) \in R_i$ of lattice paths from $\Pi$, take one with maximal $a$.
    Then $x_{a+1,b-1} = 1$, because otherwise $(a+1,b-1)$ would be a north-east corner of a lattice path of $\Pi$.
    This contradicts the minimality of $\mathbf x$, since for $\mathbf x' \coloneqq \mathbf x - e_{a+1,b-1} + e_{a,b}$, we have $\underline {\mathbf f}^\mathbf x=\underline {\mathbf f}^{\mathbf x'}$ by Lemma \ref{lem:corner-swap}, and $\mathbf x' <_{\mathrm{neglex}} \mathbf x$.
    Thus, no lattice path of $\Pi$ can have a north-east corner.
    Hence, every lattice path of $\Pi$ has at most one corner and if it has a corner, it is a south-west corner.

    If $\Pi$ contains a purely horizontal (respectively vertical) lattice path, we call its starting point (respectively endpoint) a corner as well.
    Let $C$ denote the set of all corners of lattice paths from $\Pi$ and let
    \begin{align*}
        \mathbf c \coloneqq \sum_{c \in C} e_c.
    \end{align*}
    Since $T_{A,n,\omega_i}(\mathbf c) = \mathbf x$, it remains to show that $\mathbf c \in P(\omega_i)^\Z$.
    Indeed, since the lattice paths of $\Pi$ are disjoint, any two elements of $C$ are incomparable with respect to the product order
    \[
    (a,b)\leq(c,d)
    \quad\Longleftrightarrow\quad
    a\leq c\ \text{and}\ b\leq d.
    \]
    This implies that there is no Dyck path in $H(A_n)$ containing two or more elements of $C$.
    Furthermore, Dyck paths $\mathbf p$ with $\mathrm{wt}_{\omega_i}(\mathbf p) = 0$ avoid $R_i$.
    This proves $\mathbf c\in P(\omega_i)^\Z$ and completes the proof.
\end{proof}

Now Lemma~\ref{lem:mono} for type $A$ follows.

\begin{proof}[Proof of Lemma ~\ref{lem:mono} for type $A$]
    The map $T_{A,n,\omega_i}$ is injective.
    Hence, by Lemma~\ref{lem:surjective-lattice-points}
    \begin{align*}
        \dim V_w(\widetilde\omega_{2i-1}) = |Q_{\underline{w}}(\tilde{\omega}_{2i-1})^\Z| \leq |P(\omega_i)^\Z| = \dim V(\omega_i).
    \end{align*}
    The equality now follows since $\dim V_w(\widetilde\omega_{2i-1}) = \dim V^a(\omega_i) = \dim V(\omega_i)$ by \cite[Theorem 4.1]{CILL16}.
\end{proof}

\section{Type \texorpdfstring{$C$}{C}}\label{sec-c}

We start with the setup.
Let us consider $V = \C^{2d}$ with basis $e_1,\dots,e_{2d}$ and equip it with the natural action of $\lie{sl}_{2d}(\C)$.
In $\lie{sl}_{2d}(\C)$ we fix the standard Cartan subalgebra of diagonal matrices and Borel subalgebra of upper triangular matrices.
The negative simple root operators are then $f_{j}^A = E_{j+1,j}$ as before.
We temporarily use a superscript to distinguish between types $A$ and $C$.
Next, we define a skew-symmetric and invertible matrix
\begin{align*}
    J = ((-1)^{i+1} \delta_{2d+1-i,j})_{1 \leq i,j \leq 2d} \in \C^{2d \times 2d}.
\end{align*}

With the matrix $J$ we define a type $C_d$ Lie algebra
\begin{align*}
    \lie{g}_{C_d} \coloneqq \{A \in \C^{2d \times 2d} \mid JA + A^tJ = 0\} \subseteq \lie{sl}_{2d}(\C).
\end{align*}
We again fix the Cartan subalgebra of diagonal matrices in $\lie{g}_{C_d}$ and the Borel subalgebra of upper triangular matrices.
The negative simple root operators are thus
\begin{align*}
    f_j^C = \begin{cases}
        f^A_j + f^A_{2d-j} & 1 \leq j < d \\
        f_d^A & j = d
    \end{cases}
\end{align*}
so that $\lie n_{C_d}^- \subseteq \lie n^-_{A_{2d-1}}$.
This inclusion induces a map on universal enveloping algebras which we shall denote by
\begin{align*}
    \unfold: U(\lie n_{C_d}^-) \to U(\lie n^-_{A_{2d-1}}).
\end{align*}
From now on, we omit the type superscripts when applying $\unfold$: its argument lies in the type $C$ enveloping algebra, and its image lies in the type $A$ enveloping algebra.
The above equality then reads
\begin{align}\label{unfold-formula}
    \unfold(f_j) = \begin{cases}
        f_j + f_{2d-j} & 1 \leq j < d \\
        f_d & j = d
    \end{cases}
\end{align}
The fundamental module of type $C_d$ with highest weight $\omega_i$ is the $U(\lie{g}_{C_d})$-submodule $V(\omega_i)$ of $\Lambda^i V$ generated by $e_1 \wedge \dots \wedge e_i$.
The lattice points of the string polytope are the essential exponents for the height-weighted negative lexicographic order introduced in Section \ref{sec-prel}.
For the rest of this section we use the setup with $d = 2n-1$ and set $v_i \coloneqq e_1 \wedge \dots \wedge e_{2i-1} \in V(\tilde \omega _{2i-1})$.
For fixed $n$ and $i$, we use the abbreviations
\[
T_A:=T_{A,2n-1,\omega_i},\qquad T_C:=T_{C,n,\omega_i}.
\]

\tikzset{pyramid/.style={filling, draw=none, ampersand replacement=\&}}
\newcommand{\pyramidframe}[1]{
  \pgfmathtruncatemacro{\pyrlast}{#1-1}
  \draw ([shift={(-.15em,.15em)}]M-1-1.north west)
    \foreach \pa [evaluate=\pa as \pc using {int(2*#1-\pa)}] in {1,...,#1}
      {-| ([shift={(.15em,-.15em)}]M-\pa-\pc.south east)}
    -- ([shift={(-.15em,-.15em)}]M-#1-#1.south west)
    \foreach \pa in {\pyrlast,...,1}
      {|- ([shift={(-.15em,-.15em)}]M-\pa-\pa.south west)}
    -- cycle;}
\colorlet{hookA}{teal!80!green}
\colorlet{hookB}{orange}

\begin{exa}
Let $n=3$, $i=2$, $T=T_C$, and $\mathbf x =e_{2,2} + e_{1,\overline{2}}$.
Then
\[
-\Phi_{C,3}(e_{2,2})=e_{1,2}+e_{2,2}+e_{2,3}+e_{2,\overline{2}}+e_{1,\overline{2}},
\qquad
-\Phi_{C,3}(e_{1,\overline{2}})=e_{1,\overline{2}}+e_{1,\overline{1}}.
\]
Since
\[
\mathbf t_{C,3,\omega_2}=e_{1,2}+e_{1,3}+2e_{1,\overline{2}}+e_{1,\overline{1}}
+e_{2,2}+e_{2,3}+e_{2,\overline{2}},
\]
we obtain
\[
T(\mathbf x)=\mathbf t_{C,3,\omega_2}+\Phi_{C,3}(e_{2,2})+\Phi_{C,3}(e_{1,\overline{2}})=e_{1,3}.
\]
To visualize this, we display $\mathbb R^{H(C_3)}$ as a pyramid: row $a$ contains the entries in positions $(a,b)$ with $a\le b\le\overline{a}$, and the columns are indexed by $1<2<3<\overline{2}<\overline{1}$.

\begin{align*}
\mathbf x=
\vcenter{\hbox{\begin{tikzpicture}
  \matrix (M) [pyramid] { 0 \& 0 \& 0 \& 1 \& 0 \\
                            \& 1 \& 0 \& 0 \&   \\
                            \&   \& 0 \&   \&   \\ };
  \pyramidframe{3}
  \begin{scope}[on background layer]
    \fill[ldot=hookA] (M-2-2.center) circle (.4em);
    \fill[ldot=hookB] (M-1-4.center) circle (.4em);
  \end{scope}
\end{tikzpicture}}}
&\mapsto
\vcenter{\hbox{\begin{tikzpicture}
  \matrix (M) [pyramid] { 0 \& 1 \& 1 \& 2 \& 1 \\
                            \& 1 \& 1 \& 1 \&   \\
                            \&   \& 0 \&   \&   \\ };
  \pyramidframe{3}
\end{tikzpicture}}}
-
\vcenter{\hbox{\begin{tikzpicture}
  \matrix (M) [pyramid] { 0 \& 1 \& 0 \& 2 \& 1 \\
                            \& 1 \& 1 \& 1 \&   \\
                            \&   \& 0 \&   \&   \\ };
  \pyramidframe{3}
  \begin{scope}[on background layer]
    \draw[lpath=hookA] (M-1-2.center) -- (M-2-2.center) -- (M-2-4.center)
                       -- (M-1-4.center);
    \draw[lpath=hookB] (M-1-4.center) -- (M-1-5.center);
  \end{scope}
\end{tikzpicture}}}
\\
&=
\vcenter{\hbox{\begin{tikzpicture}
  \matrix (M) [pyramid] { 0 \& 0 \& 1 \& 0 \& 0 \\
                            \& 0 \& 0 \& 0 \&   \\
                            \&   \& 0 \&   \&   \\ };
  \pyramidframe{3}
\end{tikzpicture}}}
= T(\mathbf x)
\end{align*}

We then have $\underline{\mathbf f}^{T(\mathbf x)}=f_3$ and $\underline{\mathbf f}^{T(\mathbf x)}(e_1 \wedge e_2 \wedge e_3) = e_1 \wedge e_2 \wedge e_4$.
\end{exa}

\begin{proof}[Proof of Proposition~\ref{lem:unimod} in type $C$]
Order the $e_{a,b}$ first by increasing $a$ and, for fixed $a$, by decreasing $b$ with respect to
\[
1<2<\cdots<n=\bar n<\overline{n-1}<\cdots<\bar1.
\]
By Definition~\ref{def cX}, $-\Phi_{C,n}(e_{a,b})$ contains $e_{a,b}$ with coefficient $1$.
All other terms either have first coordinate smaller than $a$, or first coordinate $a$ and second coordinate larger than $b$.
Hence the matrix of $-\Phi_{C,n}$ is upper triangular with diagonal entries $1$, and thus
\[
|\det\Phi_{C,n}|=1.
\]
\end{proof}

In this section, the main idea is to use results from Section~\ref{sec-a} for type $A_{2n-1}$ to establish Lemma \ref{lem:mono} for type $C_n$.

Let $\underline w_A$ be the reduced word of Section~\ref{sec-main} for $X=A$ with $n$ replaced by $2n-1$.
We denote the family of lowering operators for type $C$ given by the reduced expression $\underline{w}$ from Section \ref{sec-main} by $\underline{\mathbf f}$ and the family of lowering operators for type $A$ indexed with $\underline{w}_A$ by $\underline{\mathbf f}_A$.

Put
\[
\mathcal D_n=\{(a,b)\in H(A_{2n-1})\mid a+b\leq2n\}
\]
and define
\[
\iota:H(C_n)\longrightarrow \mathcal D_n,
\qquad
\iota(a,j)=(a,j),\qquad
\iota(a,\bar j)=(a,2n-j).
\]
We also apply $\iota$ to roots via their index pairs and to paths term by term.
We use the same symbol for the linear map
\[
\iota:\R^{H(C_n)}\longrightarrow\R^{H(A_{2n-1})},
\qquad
e_{a,b}\longmapsto e_{\iota(a,b)}.
\]
For $(a,b) \in H(A_{2n-1})$, let
\[
\theta(a,b)\coloneqq (2n-b,2n-a).
\]
Every $\theta$-orbit meets $\mathcal D_n$ in exactly one point.
Define $\fold:\R^{H(A_{2n-1})}\to\R^{H(C_n)}$ as the $\R$-linear map with
\[
\fold(e_{a,b})
=
\begin{cases}
e_{\iota^{-1}(a,b)},&(a,b)\in \mathcal D_n,\\
e_{\iota^{-1}(\theta(a,b))},&(a,b)\notin \mathcal D_n.
\end{cases}
\]

We shall use subscripts on FFLV polytopes and superscripts on string polytopes to specify their Lie type.
The fundamental weights in their arguments belong to the indicated type.

\begin{prop}\label{fold-compat}
For $1\leq i\leq n$ and $\mathbf x\in\R^{H(C_n)}$, we have
\[
\fold(\mathbf t_{A,2n-1,\omega_i})=\mathbf t_{C,n,\omega_i}
\qquad \text{and} \qquad
\fold\bigl(\Phi_{A,2n-1}(\iota(\mathbf x))\bigr)=\Phi_{C,n}(\mathbf x).
\]
Hence
\[
T_C(\mathbf x)
=
\fold\left(T_A(\iota(\mathbf x))\right).
\]
\end{prop}

\begin{proof}
For the first identity, the coefficient of the folded coordinate indexed by $(a,b)\in H(C_n)$ is the number of elements in the $\theta$-orbit of $\iota(a,b)$ which contain $i$ between their two indices.
This number is
\[
\fold(\mathbf t_{A,2n-1,\omega_i})_{a,b} = 
\begin{cases}
1,&i\leq b<\bar i\text{ and }a\leq i,\\
1,&b=\bar a\text{ and }a\leq i,\\
2,&\bar i\leq b<\bar a,\\
0,&\text{otherwise},
\end{cases}
\]
which is Definition~\ref{tau}.

It is enough to check the second identity on each basis vector $e_{a,b}$ for $(a,b)\in \mathcal D_n$.
Write
\[
    \beta \coloneqq \begin{cases}
    b & b \leq n \\
    \overline{2n-b} & b > n
\end{cases} \quad \text{so that} \quad (a,b)= \iota(a,\beta).
\]
If $a+b<2n$, the terms on the two sides of the fixed anti-diagonal are paired and give
\[
-\Phi_{C,n}(e_{a,\beta})
=
\sum_{\beta \leq c \leq \bar a} e_{a,c}
+
\sum_{c<a}\bigl(e_{c,\beta}+e_{c,\bar a}\bigr).
\]
If $a+b=2n$, the coordinate is fixed by $\theta$ and the paired terms give
\[
-\Phi_{C,n}(e_{a,\bar a})
=
e_{a,\bar a}+2\sum_{c<a}e_{c,\bar a}.
\]
These are the two cases of Definition~\ref{def cX}.
\end{proof}

Recall the notation $r(a,b)=a+b-1$ and $r(a,\bar j)=2n-j+a-1$ from Section ~\ref{sec-gen}.

\begin{lem}\label{lem:fold-word}
Let $\mathbf m\in\N^{H(C_n)}$.
\begin{enumerate}
    \item Every monomial summand of $\unfold(\underline {\mathbf f}^\mathbf m)$ can be written, after commuting factors $f_r,f_s$ with $|r-s|>1$, as $\underline {\mathbf f}_A^{\mathbf m'}$ in the $\underline w_A$-order, with $\fold(\mathbf m')=\mathbf m$ for some $\mathbf m' \in \N^{H(A_{2n-1})}$.
    \item Conversely, every $\mathbf m'\in\mathbb N^{H(A_{2n-1})}$ with $\operatorname{fold}(\mathbf m')=\mathbf m$ arises from some monomial summand, in particular also $\mathbf m'=\iota(\mathbf m)$.
\end{enumerate}
  
\end{lem}

\begin{proof}
Let $u = (a,b) \in H(C_n)$ and write $(a,\beta) = \iota(u)$.
Then
\[
    r(\theta \circ \iota(u)) = r(2n-\beta,2n-a) = 4n - (a+\beta) - 1 = 4n-2-r(u).
\]
Thus, by the definition of $\unfold(f_r)$ given in \eqref{unfold-formula}, we get
        \begin{align*}
            \unfold(f_{r(u)}) =
            \begin{cases}
                    f_{r(u)} + f_{r(\theta \circ \iota(u))} & b \neq \overline{a} \\
                    f_{r(u)} & b = \overline{a}
            \end{cases}.
        \end{align*}

For $b\neq\overline a$, consider the binomial expansion
\begin{align*}
    \unfold(f_{r(u)})^{m_u}=\bigl(f_{r(u)}+f_{r(\theta\circ\iota(u))}\bigr)^{m_u}.
\end{align*}
For $b=\overline a$, the cell is fixed and the corresponding factor is simply $\unfold(f_{2n-1})^{m_u}=f_{2n-1}^{m_u}$.
Two operators $f_{r(a,b)}$ and $f_{r(a',b')}$ do not commute if and only if their indices differ by one, that is, if their cells $(a,b)$ and $(a',b')$ lie on neighbouring anti-diagonals.
The operators $f_{r(u)}$ and $f_{r(\theta \circ \iota(u))}$ always commute for $u \neq (c,\overline c)$ since both are type $A$ operators indexed by non-adjacent simple roots.
Indeed, for $u \neq (c,\bar c)$ we have
\begin{align*}
    r(u) < 2n-1 < r(\theta \circ \iota(u)).
\end{align*}

Every monomial summand of $\unfold(\underline{\mathbf f}^{\mathbf m})$ is then a word consisting of one block for each $u\in H(C_n)$, with the blocks in the type $C$ order, where the block of $u$ consists of $m_u$ operators, each indexed by $\iota(u)$ or by $\theta(\iota(u))$.
Let $m'_v$ be the number of operators in this monomial summand indexed by $v\in H(A_{2n-1})$.
Then we have $\fold(\mathbf m')=\mathbf m$.
Conversely, let $\operatorname{fold}(\mathbf m')=\mathbf m$.
For $u=(a,b)$ with $b\neq\overline a$, choose in the block of $u$ the operator $f_{r(u)}$ exactly $m'_{\iota(u)}$ times and $f_{r(\theta\circ\iota(u))}$ the remaining $m_u-m'_{\iota(u)}=m'_{\theta(\iota(u))}$ times.
The resulting summand has exponent $\mathbf m'$.
Using the rearrangement from (i) below, this shows (ii).

To show (i), we use a bubble sort procedure:
If some pair of operators occurs in the opposite order to that in the $\underline{w}_A$ order, then some adjacent pair does, and interchanging it decreases the number of such pairs.
It remains to show that the two operators in each such adjacent pair commute.
That is, operators whose cells lie on neighbouring anti-diagonals are already in the correct order.

An operator of $\unfold(\underline{\mathbf f}^{\mathbf m})$ indexed by $v \in H(A_{2n-1})$ occupies the position of its representative $\hat v\in\{v,\theta(v)\}\cap\mathcal D_n$, and under $\iota$ the type $C$ order is the order $\preceq$ restricted to $\mathcal D_n$.
Hence it suffices to show for cells $v=(a,b)$ and $w=(c,d)$ in different $\theta$-orbits on neighbouring anti-diagonals in $H(A_{2n-1})$ (i.e. $|(a+b)-(c+d)|=1$) that
\begin{align*}
    v\prec w\iff\hat v\prec\hat w.
\end{align*} 
This is clear if $v,w\in\mathcal D_n$, so let $v\notin\mathcal D_n$, that is, $a+b>2n$.

If $w\in\mathcal D_n$, then $a+b=2n+1$ and $c+d=2n$, so
\begin{align*}
    \hat w=w=(c,2n-c) \quad \text{and} \quad \hat v=\theta(v)=(a-1,2n-a).
\end{align*}
Both $v \prec w$ and $\theta(v) \prec w$ are equivalent to $c \leq a-1$.

If $w \notin \mathcal D_n$, we may assume $c+d=a+b+1$, since the claim is symmetric in $v$ and $w$.
Then $v \prec w \iff b < d \iff c \leq a$, and
\begin{align*}
    \theta(v)=(2n-b,2n-a) \prec \theta(w)=(2n-d,2n-c) \iff c \leq a
\end{align*}
as well.
\end{proof}

\begin{exa}\label{exa:C2-folding}
Let $n=2$ and $i=2$.
The reduced expression given in Section ~\ref{sec-main} is
\[
\underline w=s_3s_2s_3s_1.
\]
For the lattice point $\mathbf x=e_{1,2}\in P_{C_2}(\omega_2)^\Z$ one has
\[
\iota(\mathbf x)=e_{1,2}\in P_{A_3}(\omega_2)^\Z.
\]
The two affine maps give
\[
T_A(\iota(\mathbf x))=e_{2,2}+e_{2,3},
\qquad
T_C(\mathbf x)=e_{1,2}+e_{2,2},
\]
and indeed
\[
\fold(e_{2,2}+e_{2,3})=e_{1,2}+e_{2,2}.
\]
On the monomial level,
\[
\underline {\mathbf f}_A^{\,e_{2,2}+e_{2,3}}=f_4f_3,
\qquad
\underline {\mathbf f}^{\,e_{1,2}+e_{2,2}}=f_2f_3,
\]
while
\[
\unfold(f_2f_3)=(f_2+f_4)f_3.
\]
Thus the type $A_5$ monomial $f_4f_3$ occurs as one of the summands of the unfolded type $C$ monomial.
\end{exa}

\begin{lem}\label{unfold-positive}
For $\mathbf m\in \N^{H(C_n)}$, all coefficients of
\[
\unfold(\underline {\mathbf f}^\mathbf m)
(e_1 \wedge \cdots \wedge e_{2i-1})
\]
in the ordered wedge basis are non-negative integers.
\end{lem}

\begin{proof}
The expansion of $\unfold(\underline {\mathbf f}^\mathbf m)$ has non-negative integral coefficients.
Every type $A$ simple lowering operator is $E_{r+1,r}$ in the basis $(e_1,\ldots,e_{4n-2})$.
On an ordered wedge it either acts as zero or replaces $e_r$ by $e_{r+1}$.
In the non-zero case no reordering is required, so the coefficient is $1$.
\end{proof}

\begin{lem}\label{C-to-A-FFLV}
For $1\leq i\leq n$,
\[
\mathbf x \in P_{C_n}(\omega_i)^\Z
\quad\Longrightarrow\quad
\iota(\mathbf x)\in P_{A_{2n-1}}(\omega_i)^\Z.
\]
\end{lem}

\begin{proof}
    It is clear from the definition that $\iota$ maps integral points to integral points.
    Hence, it suffices to show that $\iota(\mathbf x) \in P_{A_{2n-1}}(\omega_i)$.
    Let $\mathbf p = (p(0),\dots,p(\ell))$ be a type $A_{2n-1}$ Dyck path.
    If $p(0) = \alpha_r$ with $r > n$, then for every root $\alpha_{a,b}$ in $\mathbf p$, one has $\iota(\mathbf x)_{a,b} = 0$ since $(a,b) \notin \mathcal D_n$ and the inequality is immediate.
    Now we assume $r \leq n$.
    
    If $p(\ell) = \alpha_k$ for $k \leq n$, then $\mathbf p$ is contained inside of $\mathcal D_n$ and hence $\iota^{-1}(\mathbf p)$ is a type $C_n$ Dyck path.
    Since $\mathbf x \in P_{C_n}(\omega_i)$, we have
        \begin{align*}
            \sum_{s \in \mathbf p} \iota(\mathbf x)_{s} = \sum_{s \in \mathbf \iota^{-1}(\mathbf p)} \mathbf x_s \leq \mathrm{wt}_{\omega_i}(\mathbf p).
        \end{align*}
        
    If $p(\ell) = \alpha_k$ for $k > n$, then $\mathbf p$ crosses the anti-diagonal $a+b = 2n$.
    Hence there is a $0 \leq j \leq \ell$ such that $p(j) = \alpha_{c, 2n-c} = \iota(\alpha_{c,\overline c})$ with $1 \leq c \leq n$.
    Since $\mathbf p$ can only go right or down, the steps $p(j+1),\dots,p(\ell)$ are all outside of $\mathcal D_n$.
    The truncated Dyck path
    \begin{align*}
        \mathbf p' \coloneqq (\iota^{-1}(p(0)),\dots,\iota^{-1}(p(j)))
    \end{align*}
    is a type $C_n$ Dyck path with $\mathrm{wt}_{\omega_i}(\mathbf p') = \mathrm{wt}_{\omega_i}(\mathbf p)$ since both weights are $1$ exactly when $r \leq i$.
    Since $\mathbf x \in P_{C_n}(\omega_i)$, we thus have
    \begin{align*}
        \sum_{s \in \mathbf p} \iota(\mathbf x)_s = \sum_{s \in \mathbf p'} \mathbf x_s \leq\mathrm{wt}_{\omega_i}(\mathbf p') = \mathrm{wt}_{\omega_i}(\mathbf p).
    \end{align*}
    This shows that $\iota(\mathbf x) \in P_{A_{2n-1}}(\omega_i)$.
\end{proof}

\begin{lem}\label{fold-order}
Let $\mathbf y\in P_{C_n}(\omega_i)^\Z$ and
\[
\mathbf x \coloneqq T_A(\iota(\mathbf y)).
\]
If $\mathbf m\in\{0,1\}^{H(A_{2n-1})}$ satisfies $\underline{\mathbf f}_A^\mathbf m v_i\neq0$ and has the same $A_{4n-3}$-weight as $\mathbf x$, then
\[
\fold(\mathbf x)\leq_{\mathrm{neglex}}\fold(\mathbf m).
\]
\end{lem}

\begin{proof}
    Combining Proposition~\ref{prop:restrA} and Lemma~\ref{lem:path-family} yields that $\mathbf m$ is a $0$--$1$ filling of $R_i$ with zero set given by disjoint lattice paths.
    Suppose one of the lattice paths has a north-east corner.
    Choose such a corner $(a,b)$ with maximal $a$.
    Then, as in the proof of Lemma~\ref{lem:surjective-lattice-points}, $m_{a+1,b-1} = 1$ and for $\mathbf m' \coloneqq \mathbf m - e_{a+1,b-1} + e_{a,b}$, we have $\underline{\mathbf f}_A^{\mathbf m} = \underline{\mathbf f}_A^{\mathbf m'}$ by Lemma~\ref{lem:corner-swap}.
    We claim that
    \begin{align}\label{eq:minimizing}
        \fold(\mathbf m') <_{\mathrm{neglex}} \fold(\mathbf m).
    \end{align}
    To prove this, note that either both of $(a+1,b-1)$ and $(a,b)$ are in $\mathcal D_n$ or both are outside of $\mathcal D_n$.
    If both are in $\mathcal D_n$, then the claim holds since $(a+1,b-1) \prec (a,b)$.
    If both are outside $\mathcal D_n$, then
    \begin{align*}
        \theta(a+1,b-1) = (2n-b+1,2n-a-1) \prec (2n-b,2n-a) = \theta(a,b)
    \end{align*}
    which finishes the proof of the claim (\ref{eq:minimizing}).
    We repeat this until we get a $0$--$1$ vector $\mathbf h$ without north-east corners satisfying $\underline{\mathbf f}_A^{\mathbf h} = \underline{\mathbf f}_A^{\mathbf m}$ and $\fold(\mathbf h) \leq_{\mathrm{neglex}} \fold(\mathbf m)$.
    Using the argument from the proof of Lemma~\ref{lem:surjective-lattice-points}, we obtain that
    \begin{align*}
        \mathbf h \in T_A(P_{A_{2n-1}}(\omega_i)^\Z) = Q^{A_{4n-3}}_{\underline{w}_A}(\tilde \omega_{2i-1})^\Z
    \end{align*}
    and hence $\mathbf h$ is essential for $\underline{\mathbf f}_A$ acting on $v_i$ by Lemma~\ref{lem:truncated-essential}.
    Similarly, by Lemma~\ref{C-to-A-FFLV}, $\iota(\mathbf y) \in P_{A_{2n-1}}(\omega_i)^\Z$ and thus $\mathbf x \in Q_{\underline{w}_A}^{A_{4n-3}}(\tilde \omega_{2i-1})^\Z$ is essential as well.
    Since both $\mathbf h$ and $\mathbf x$ are essential of the same weight, we have $\mathbf h = \mathbf x$ by Remark~\ref{rem:one-dimensional} since $V(\tilde \omega_{2i-1})$ is minuscule in type $A_{4n-3}$.
    This finishes the proof since
    \begin{align*}
        \fold(\mathbf x) = \fold(\mathbf h) \leq_{\mathrm{neglex}} \fold(\mathbf m).
    \end{align*}
\end{proof}

\begin{proof}[Proof of Lemma~\ref{lem:mono} in type $C$]
Fix $1\leq i\leq n$ and let $\mathbf y\in P_{C_n}(\omega_i)^\Z$.
Set
\[
\mathbf y^A=\iota(\mathbf y),\qquad
\mathbf x=T_A(\mathbf y^A),\qquad
\mathbf m=T_C(\mathbf y).
\]
By Proposition~\ref{fold-compat},
\[
\mathbf m=\fold(\mathbf x).
\]
By Lemma~\ref{C-to-A-FFLV} and the type $A$ case,
\[
\mathbf x\in Q_{\underline w_A}^{A_{4n-3}}(\widetilde\omega_{2i-1})^\Z.
\]
Hence $\underline {\mathbf f}_A^\mathbf x v_i\neq0$, and Proposition~\ref{prop:restrA} implies that $\mathbf x$ is a $0$--$1$ vector.

By Lemma~\ref{lem:fold-word}, one of the words in the formal expansion of $\unfold(\underline {\mathbf f}^\mathbf m)$ can be transformed into $\underline {\mathbf f}_A^\mathbf x$ by commuting only non-adjacent simple-root operators.
Since $\underline {\mathbf f}_A^\mathbf x v_i\neq0$, Lemma~\ref{unfold-positive} shows that this contribution cannot cancel.
Hence
\[
\underline {\mathbf f}^\mathbf m v_i\neq0.
\]
It remains to prove essentiality.
Let
\[
\mu=\operatorname{wt}_{A_{4n-3}}(\underline {\mathbf f}_A^\mathbf x v_i).
\]
Suppose that
\[
\underline {\mathbf f}^\mathbf m v_i
\in
\operatorname{span}\{\underline {\mathbf f}^\mathbf z v_i\mid \mathbf z<_{\mathrm{neglex}}\mathbf m\}.
\]
Unfold this relation and project to the $A_{4n-3}$-weight space of weight $\mu$.
The left-hand side has a non-zero $\mu$-component, since $\underline {\mathbf f}_A^\mathbf x v_i$ occurs in it.
Hence, for some $\mathbf z <_{\mathrm{neglex}} \mathbf m$, there is a monomial summand of $\operatorname{unfold}(\underline{\mathbf f}^\mathbf z)$ whose action on $v_i$ is non-zero and has $A_{4n-3}$-weight $\mu$.
By Lemma~\ref{lem:fold-word}, after commuting factors it can be written as $\underline {\mathbf f}_A^\mathbf w$ in the $\underline w_A$-order with $\fold(\mathbf w) = \mathbf z$.
Since the ambient type $A$ fundamental representation is minuscule, $\mathbf w$ is a $0$--$1$ exponent.
Lemma~\ref{fold-order} gives
\[
\mathbf m =\fold(\mathbf x) \leq_{\mathrm{neglex}} \fold(\mathbf w) = \mathbf z,
\]
contrary to $\mathbf z <_{\mathrm{neglex}} \mathbf m$.
Thus $\mathbf m$ is essential for the truncated ordered family attached to $\underline w$.
Lemma~\ref{lem:truncated-essential} therefore gives
\[
T_C(\mathbf y) = \mathbf m
\in Q_{\underline w}^{C_{2n-1}}(\widetilde\omega_{2i-1})^\Z.
\]
This proves that
\begin{align*}
    T_C(P_{C_n}(\omega_i)^\Z) \subseteq Q_{\underline{w}}^{C_{2n-1}}(\tilde \omega_{2i-1})^\Z.
\end{align*}
Equality now follows by an argument similar to that used in type $A$, since
\begin{align*}
    \dim V_w(\widetilde\omega_{2i-1}) = \dim V^a(\omega_i) = \dim V(\omega_i)
\end{align*} holds for type $C$ as well by \cite[Theorem 7.1]{CILL16}.
\end{proof}

\printbibliography
\end{document}